\documentclass[12pt,reqno]{amsart}
\usepackage{amsmath}
\usepackage{amssymb,amscd}
\usepackage{multicol}

\numberwithin{equation}{section}

\newtheorem{theorem}{Theorem}[section]
\newtheorem{corollary}[theorem]{Corollary}
\newtheorem{lemma}[theorem]{Lemma}

\theoremstyle{definition}
\newtheorem{defn}[theorem]{Definition}
\newtheorem{remark}[theorem]{Remark}

\def \({\left(}
\def \){\right)}
\def \<{\langle}
\def \>{\rangle}

\begin{document}
%\titlerunning{SKT MANIFOLDS}
\title[SKT manifolds]{A construction of SKT manifolds using toric geometry}

\author[H. Coban]{Hatice Coban}

\address{Department of Mathematics and Computer Science, University of Antwerp imec-IDlab, Antwerp, Belgium} 

\email{hatice.coban@uantwerpen.be}

\author[C. Haciyusufoglu]{Cagri Haciyusufoglu}
 \address{ Mathematics Group, Middle East Technical University, Northern Cyprus Campus, Guzelyurt, Mersin 10, Turkey} \email{hcagri@metu.edu.tr} 
 
\author[M. Poddar]{Mainak Poddar} 

\address{Mathematics Department, Indian Institute of Science Education and Research, Pune, India}

\email{mainak@iiserpune.ac.in}



\begin{abstract} We produce infinite families of SKT manifolds by using methods of toric geometry like the 
	$J$-construction. These SKT manifolds are total spaces of certain principal $G$-bundles over smooth projective toric varieties, where $G$ is an even dimensional compact connected Lie group.
\end{abstract}

%\support{The research of the third-named author has been supported in part by a SERB MATRICS Grant, MTR/2019/001613, and an SRP grant from METU NCC.  }

\maketitle

\section{Introduction}
 
 An SKT structure  on a complex manifold is a generalization of the more familiar notion of K\"ahler structure in the following precise sense.
Consider a Hermitian manifold $E$  with Riemannian metric $g$ and integrable complex structure $\mathcal{J}$. Consider a tangential connection $\nabla$ on $E$ which is compatible with the Hermitian structure and define the associated torsion $3$-tensor
 $\mu$ by 
 $$\mu(A,B,C) = g(T(A,B), C) \, ,$$ where $T$ denotes the torsion $2$-form of $\nabla$, and $A$, $B$, and $C$ are arbitrary vector fields on $E$.
 
 Then $M$ admits a unique Hermitian connection $\nabla$ such that the torsion tensor $\mu$ 
  is totally skew symmetric. This connection is known as  the  KT (K\"ahler with torsion) or Bismut connection. Bismut used it to prove a local index formula for the Dolbeault operator when the complex manifold is not K\"ahler \cite{Bis}. 

For the KT connection, we actually have 
$$ \mu(A,B,C) = dF( \mathcal{J} A, \mathcal{J} B, \mathcal{J} C)\,,$$ where $F$ denotes the fundamental $2$-form defined
 by $F(A,B) = g( \mathcal{J} A, B)$. Note that $M$ is K\"ahler if $dF =0 $, or equivalently, if $\mu = 0$. 
  A KT connection is called  strong KT or SKT if $d\mu = 0 $ or equivalently, if 
 $dd^c F = 0$ (or $\partial\overline{\partial} F = 0$). Thus, we may say that $E$ admits
 an SKT structure if it admits a Hermitian metric whose fundamental form is $dd^c$-closed or, equivalently,
 $\partial\overline{\partial}$-closed. Such a metric is referred to as an SKT metric.

 SKT  structures arise naturally in two dimensional sigma models with $(4,0)$ supersymmetry in physics \cite{HP}.  
  They are also closely related to generalized K\"ahler geometry 
 (cf. \cite{Hit,Gua,FT}). A generalized K\"ahler structure is a pair of SKT structures with a certain compatibility condition. 

Various  constructions and  properties of SKT manifolds have been studied extensively (cf. \cite{Gaud,FG,FPS,MS}).
 In particular, Grantcharov et. al. \cite{ggp} have given a construction of  SKT structures on the 
 total space of  torus principal bundles of even rank, say $2k$, over a complex K\"ahler manifold
 under the condition
 \begin{equation}\label{sq}
  \sum_{j =1}^{2k} w_j^2 = 0  
 \end{equation}
 where $w_j$'s are certain characteristic classes (see Section \ref{metric}).  When the base manifold of the principal bundle is of dimension $4$, then the
 sufficient condition \eqref{sq} may be dealt with using the intersection form on the middle cohomology of the base.  In \cite{ggp} some examples of SKT manifolds of dimension six are given using this principle.  However, the condition is not easy to verify in general. In fact, nontrivial torus principal bundles that satisfy   \eqref{sq} seem to be relatively rare. 
  
   In these notes, we give infinite families of examples  of nontrivial torus bundles over manifolds of arbitrary  dimension that satisfy the condition \eqref{sq}. Our main tools are from toric geometry.
   We observe that \eqref{sq} holds for certain torus 
   bundles over every Bott manifold.  The main point is that these manifolds admit nontrivial line bundles whose first Chern class has square zero. Moreover, some of these examples may be used to produce an infinite family of examples by applying the $J$-construction in toric geometry. Finally, using some 
   results of \cite{PT}, we produce more SKT manifolds from each of the above examples by extension of structure group.

\section{Smooth toric varieties}\label{stv}

The theory of toric varieties and their topological counterparts are widely studied and there are many good references like \cite{Ful,Oda,BP} etc. for them. The main purpose of this section is to introduce some definitions and notations. 

A complex toric variety $X$ of complex dimension $n$ is a (partial) compactification of the algebraic torus
 $T^n_{\mathbb{C} }=(\mathbb{C}^{\ast})^n$, which admits a natural extension of the 
 translation action of  $T^n_{\mathbb{C}} $ on itself. We will only deal with toric varieties that are smooth (nonsingular) and compact (complete). A projective toric variety is, of course, compact.  For a smooth toric variety,
 the action of  $T^n_{\mathbb{C}}$ in a neighborhood of any fixed point is the same, up to an automorphism, as its standard action on $(\mathbb{C}^{\ast})^n$.
 
  A key role is played in toric geometry by the  $T^n_{\mathbb{C}}$-invariant subvarieties of $X$
  of complex codimension one, known as the torus invariant divisors. There are finitely many of them and let these be $D_1, \ldots, D_m$. 
   Under natural choice of orientations induced by the complex structure, the  stabilizer of such a divisor $D_i$ can be identified with a vector $\lambda_i$ in the co-character lattice $N \cong \mathbb{Z}^n$ of complex one-parameter  subgroups of  $T^n_{\mathbb{C}}$.  Topologists
  refer to $\lambda_i$ as a 
  (directed) characteristic vector. One can form an $n \times m$ matrix 
  $$\Lambda = [\lambda_1, \ldots, \lambda_m] $$ with the characteristic vectors of 
  the invariant divisors, in some fixed order, as columns. Such a matrix will be called a characteristic matrix.  
  
  We denote the coordinates of the characteristic vector $\lambda_i$ by $\lambda_{ji}$, $ 1\le j \le n$.
  Let $w_i$ denote the  first Chern class of the complex line bundle $L_i$ that corresponds to $D_i$ under the divisor-line bundle correspondence. Then the classes $w_1, \ldots, w_m$ generate the integral cohomology ring of the smooth toric variety $X$. The linear relations among the $w_j$'s are encoded in the rows of the corresponding characteristic matrix.

  Each characteristic vector $\lambda_i$ generates a ray in $N \otimes \mathbb{R}$, which we will still denote
 by ${\lambda_i}$ for notational simplicity. These give rise to a combinatorial gadget called the fan, denoted by $\Sigma$, of the toric variety. The fan is a collection of cones. The zero vector forms the unique $0$-dimensional cone in the fan.
 Rays corresponding to characteristic vectors constitute the $1$-dimensional cones of the fan. More generally a collection $\lambda_{i_1}, \ldots,  \lambda_{i_k} $ of characteristic vectors generate a $k$-dimensional cone  in the fan if the corresponding torus invariant divisors have nonempty intersection. We denote such a cone by 
 $ \langle \lambda_{i_1}, \ldots,  \lambda_{i_k} \rangle $.
  The fan completely determines the  geometry of the toric variety. For instance, the singular cohomology ring of a smooth compact toric variety $X$ with integer coefficients (cf. \cite[pp. 134]{Oda}) is given by 
 \begin{equation}\label{cohring}
  H^{\ast}(X) = \mathbb{Z}[w_1, \ldots, w_m] / \mathcal{I} 
  \end{equation}
  where the ideal $\mathcal{I}$ is generated by 
  \begin{enumerate}
  	\item all products $w_{i_1}\cdots w_{i_k}$ such that $\lambda_{i_1}, \ldots, \lambda_{i_k}$  do not form a cone of $\Sigma$, and
  	\item all linear combinations $\sum_{i =1}^m  \lambda_{ji} w_i$  where  $1 \le j \le n$.
  \end{enumerate}

  \section{Bott manifolds}\label{bm}
  
  A Bott tower (cf. \cite{BS,GK})  of height $t$,  
  $$ M_t \to M_{t-1} \to \ldots \to M_2 \to M_1 \to M_0 = \{ \rm{point} \} \,,$$  
  is an iterated projective bundle such that each $M_k$ is a 
  $\mathbb{P}^1$-bundle over   $M_{k-1}$. More precisely, 
  $M_k = \mathbb{P}( \mathcal{O}_{M_{k-1}} \oplus \mathcal{L}_{k-1})$ where $\mathcal{L}_{k-1}$ is a line bundle over $M_{k-1}$.  Each $M_k$ is a smooth projective toric variety of complex dimension $k$ and is known as a Bott manifold.  Note that $M_1$ is just $\mathbb{P}^1$ and $M_2$ is a Hirzebruch surface.
  
   We now describe the fan of $M_k$ for $k \ge 1$. Let $e_1,
  \ldots, e_k$ denote the standard basis of $\mathbb{Z}^k$. 
    We identify $e_1, \ldots, e_{k-1}$ with the standard basis of $\mathbb{Z}^{k-1}$ without confusion. 
  The manifold $M_k$ in a Bott tower has $2k$ characteristic vectors which we denote by $\lambda^{(k)}_1, \ldots, \lambda^{(k)}_{2k}$.
   These may be defined inductively as follows: Start with 
   $\lambda_1^{(1)} = e_1$ and $\lambda_2^{(1)} = - e_1$.  
   Then define
  $$\lambda_i^{(k)} = \lambda_i^{(k-1)} \quad  {\rm and} \quad \lambda^{(k)}_{k+i} =  \lambda^{(k-1)}_{k-1+i}  + c_{i,k} e_k \quad
   {\rm for} \quad  i = 1, \ldots, k-1 \,.$$
   Moreover, define $$\lambda_k^{(k)} = e_k \quad {\rm and}  \quad \lambda_{2k}^{(k)} = - e_k \,.$$    Here each $c_{i,k}$ is an integer. It is to be noted that different values of these constants may produce different manifolds $M_k$. 
   
      With respect to the standard basis, the characteristic matrix of $M_k$ has the following  form. 
   
  \begin{equation}\label{lamk} \Lambda(M_k) =    \left[  \begin{array}{cccccccc}
1 & 0 & \ldots & 0 & -1 & 0&\ldots & 0\\
0 & 1 & \ldots & 0 & c_{1,2} &  -1 & \ldots & 0 \\ 
\ldots & \ldots & \ldots & \ldots & \ldots & \ldots & \ldots & \ldots \\
0 & 0 & \ldots & 1 & c_{1,k} & c_{2,k} & \ldots & -1 \\
\end{array} \right]    \end{equation}

   There are $2^k$  cones in the fan of $M_k$ of the top dimension $k$. These are also  easy to identify inductively. The top dimensional cones for $M_1$ are simply $\langle e_1 \rangle$ and
    $\langle -e_1 \rangle$.
    Corresponding to each $(k-1)$-dimensional cone 
    $$\sigma = \langle \lambda^{(k-1)}_{i_1}, \ldots,  \lambda^{(k-1)}_{i_{k-1} } \rangle$$
    	 in the fan of $M_{k-1}$, there are 
   two $k$-dimensional cones, namely $\langle \widetilde{\sigma}, e_k \rangle$ and  $\langle \widetilde{\sigma}, - e_k \rangle$, in the fan of $M_k$ where
   $$\widetilde{\sigma} = \langle \lambda^{(k)}_{\tilde{i}_1}, \ldots,  \lambda^{(k)}_{\tilde{i}_{k-1}}  \rangle \, \quad {\rm and} \; \tilde{i}_j = \left\{ \begin{array}{ll} i_j & {\rm if}\; i_j <k \\
                                                                         i_j + 1 & {\rm if}\; i_j \ge k 
                                                                         \end{array} \right.   $$

   \begin{lemma}\label{linesquare}
   		There exists a nontrivial line bundle on $M_k$  whose first Chern class has square zero, for any $k\ge 1$. There exist at least two such independent line bundles on $M_k$ if $k \ge 2$.  
   \end{lemma}

 \begin{proof}  Let $L_j$ denote the holomorphic line bundle associated to the torus invariant divisor of $M_k$ corresponding to $\lambda_j^{(k)}$. Let $w_j$ be the 
 first Chern class of $L_j$. By \eqref{cohring}, the abelian group
 	$H^2(M_k)$ is freely generated by $\{ w_j \mid j=1,\ldots, k \}$.

 	 We also have the following linear relations corresponding to the rows of the characteristic matrix
 	 \eqref{lamk}, 
 	\begin{equation}\label{weqs}
 	\begin{array}{l}
 	w_{k+1} = w_1 \\

 	w_{k+2} = w_2 + c_{1,2} w_{k+1} \\  
 	
 	w_{k+3} = w_3 + c_{1,3} w_{k+1} + c_{2,3} w_{k+2} \\  
 	
 	\ldots \\
 	
 	w_{2k} = w_k + c_{1,k} w_{k+1} + c_{2,k} w_{k+2} + \ldots + c_{k-1,k} w_{2k-1} \,.
 	
 	\end{array}
 	\end{equation}

 	 We note that $\lambda_i^{(k)}$ and $\lambda_{k+i}^{(k)}$ never form a cone in the fan of  $M_k$. This  implies that    $w_i w_{k+i} = 0$ for  $i=1, \ldots, k$. 
 	 Then using the first equation in \eqref{weqs},  observe that  
 	 $$ w_1^2 =  w_1 w_{k+1} = 0 \,.$$
 	
 	 Next, assume $k \ge 2$. Then multiplying the second equation of \eqref{weqs} by $w_2$ and using $w_{k+1} = w_1$, we get 
 	 $$ w_2^2 = - c_{1,2} w_1 w_2  \,.$$  
 	It follows that $(x w_1 + y w_2)^2 = 0$ if $2x = c_{1,2} y$. 
 	Hence, $(c_{1,2} w_1 + 2 w_2)^2 = 0$. Thus, the line bundles $L_1$ and 
 	$ L_1^{ c_{1,2}} \otimes L_2^2$ over $M_k$ satisfy the desired property. 
 \end{proof}

 \begin{remark} Proceeding inductively in the above proof, we can express $w_j^2$ as a linear combination of the $w_i w_j$'s with $i<j$. Moreover, the classes $w_i w_j$ with $i< j \le k$
 	form a basis of $H^4(M_k)$.
 		However, it is necessary to impose conditions on the defining constants $c_{i,j}$'s of $M_k$ to ensure existence of further independent line bundles whose first Chern class squared is zero. For instance,  $(x w_1 + y_2 + z w_3)^2 = 0$ if and only if 
 	$$2x = c_{1,2} \, y, \quad 2y =  c_{2, 3}\, z , \quad {\rm and} \quad  2x = (c_{1,3} +  c_{1,2}\, c_{2,3}) \, z\,.$$ 
 	In addition, if $z \neq 0$, it follows that  $2 c_{1,3} + c_{1,2}\, c_{2,3} = 0$ is required. 
 \end{remark}

 \section{$J$-construction on nonsingular toric varieties}\label{jc}
 
 The origins of the $J$-construction can be traced back to \cite{PB} and \cite{KK} who referred to it as {\em simplicial wedge}  and {\em dual wedge} respectively. It was used by Ewald in
 \cite{ewald}, a work we will invoke below, who called it {\em  canonical extension}. It was referred to as the {\em doubling construction} in \cite{LdM}.  But it was introduced in its current avatar in \cite{BBCG}. The relevance of this technique to us is that it produces a projective toric variety of higher dimension starting from one of lower dimension, and the cohomology ring generators of the two varieties have a simple bijective correspondence.  
 
Let $X$ be a toric variety of dimension $n$ with $m$ torus invariant divisors. Let $J$ be an $m$-tuple of natural numbers, $J = (j_1, \ldots, j_m)$.
 A $J$-construction on  $X$ produces a toric variety $X(J)$ of dimension 
 $$d_J := n+ \sum_{k=1}^m (j_k - 1) \,.$$ 
 By definition, $X(J) = X$ when $J = (1, \ldots, 1)$.  In general, $X(J)$ can be obtained from $X$ by a sequence of {\em atomic} steps that increase $d_J$ by one. More precisely, such an atomic 
 step produces 
 $X(J)$ for $J = (j_1, \ldots, j_i, \ldots, j_m)$, where $j_i \ge 2$,  from $X(J')$ where $J' = (j_1, \ldots, j_i -1, \ldots, j_m) $. The order in which these atomic steps are performed is not important for the end result. An atomic step  is also commonly referred to as a {\em simplicial wedge construction}. 
  
  Without loss of generality, we demonstrate how to perform the simplicial wedge construction to produce
   $X(J)$, where $J=(2,1,\ldots,1)$, from $X$. We may conveniently refer to this as performing the simplicial wedge construction along the characteristic vector $\lambda_1$. 
 In  the language of fans, it amounts to building a complete fan $\Sigma (J) $ of one 
 higher dimension than $\Sigma$. To accomplish this, the fan $\Sigma$, with the exception of $\lambda_1$, is embedded in a coordinate hyperplane $x_{n+1}=0$ of $\mathbb{R}^{n+1}$, the support of  $\Sigma(J)$. 
 Naturally, we define ${\lambda_i}(J) = (\lambda_i,0)$ for $1 < i \le m$.
 The vector $\lambda_1$ is modified to ${\lambda}_1(J) = (\lambda_1, -1)$. The fan $\Sigma(J)$ has one additional characteristic vector, ${\lambda}_{m+1}(J) = (0,\ldots, 0,1)$.
 
  We may write $\widetilde{\lambda}_i$ for $\lambda_i(J)$ without confusion for notational simplicity. Moreover, we use the notation 
 $\lambda_i = (\lambda_{1i}, \ldots, \lambda_{ni})$ as in Section \ref{stv}.
  Thus, we have the following characteristic 
 matrix for  ${\Sigma}(J)$.
 $$ {\Lambda}(J) = \left[  \begin{array}{ccccc}
  \lambda_{11} & \lambda_{12} & \ldots & \lambda_{1m} & 0 \\
  \lambda_{21} & \lambda_{22} & \ldots & \lambda_{2m} & 0\\
  \ldots & \ldots & \ldots & \ldots & \ldots \\
  \lambda_{n1} & \lambda_{n2} & \ldots & \lambda_{nm} &0 \\
  -1 & 0 & \ldots & 0 & 1 \end{array}
 \right] $$
 
  The top or ($n+1$)-dimensional cones of the fan ${\Sigma}(J)$ are as follows. For every $n$-dimensional cone 
  $<\lambda_{i_1} , \ldots, \lambda_{i_n}>$ of $\Sigma$ that does not contain $\lambda_1$, there are two 
 $(n+1)$-dimensional cones in ${\Sigma}(J)$ namely, $<\widetilde{\lambda}_{i_1}, \ldots, \widetilde{\lambda}_{i_n}, \widetilde{\lambda}_{1} >$  and $<\widetilde{\lambda}_{i_1}, \ldots, \widetilde{\lambda}_{i_n}, \widetilde{\lambda}_{m+1} >$. 
 Moreover, for every $n$-dimensional cone that contains $\lambda_1$, say $ <\lambda_1, \lambda_{i_2}, \ldots, \lambda_{i_n} > $,  there is 
 an  $(n+1)$-dimensional cone in ${\Sigma}(J) $, 
 $< \widetilde{\lambda}_{1}, \widetilde{\lambda}_{i_2}, \ldots,
 \widetilde{\lambda}_{i_n}, \widetilde{\lambda}_{m+1} >$. 
 The subcones of these top dimensional cones constitute, and thus determine, the fan ${\Sigma}(J)$.
 
 The toric variety ${X}(J)$ corresponding to ${\Sigma}(J)$ is nonsingular if $X$ is so. We denote the torus invariant divisor of ${X}(J)$ corresponding to ${\lambda_j}(J)$ by $D_j(J) $,  and its first Chern class by $w_j(J)$, or simply by $\widetilde{w}_j$ when there is no scope for confusion. We can read off the linear relations among the generators $\widetilde{w}_j$'s from the rows of the matrix ${\Lambda}(J)$.
 Note that $\widetilde{w}_1 = \widetilde{w}_{m+1}$. It follows that the classes $\widetilde{w}_1, \ldots, \widetilde{w}_m$ generate the cohomology ring of $X(J)$. This easily generalises to the case of an arbitrary 
 $J$ by induction.

%  Moreover, every $\widetilde{\lambda_i}$ is in an $(n+1)$-dimensional cone that contains  $\widetilde{\lambda}_1$ or $\widetilde{\lambda}_{m+1}$ or both. In particular for any $n$-dimensional cone $<\lambda_{i_1}, \ldots, \lambda_{i_n}> $ of $\Sigma$, the intersection number $\widetilde{D}_1 \widetilde{D}_{i_1} \ldots \widetilde{D}_{i_n} $ is one and the corresponding cohomology class   $  \widetilde{w}_{1} \widetilde{w}_{i_1} \cdots  \widetilde{w}_{i_n} $ is a generator of $H^{2(n+1)}(X(J))$.
 
 The notations used above may be conveniently applied to the situation of a general $J$ as well. We refer the reader to \cite[Section 3]{BBCG} for a more comprehensive description of the $J$-construction.
  We content ourselves here by providing below the characteristic matrix $\Lambda(J)$ when $J = (3,2,1,\ldots, 1)$ as an example: 
 
 %\begin{equation}\label{lambdaJ}
 $$ {\Lambda}(J) = \left[  \begin{array}{cccccccc}
\lambda_{1} & \lambda_{2} & \lambda_3 & \ldots & \lambda_{m} & \vec{0} & \vec{0}  & \vec{0}\\
-1 & 0 & 0  & \ldots & 0 & 1 & {0}  & 0 \\ 
-1 & 0 & 0  & \ldots & 0 & 0 & 1  & 0 \\
0 & -1 & 0  & \ldots & 0 & 0 & 0 & 1 \\
\end{array} \right] $$
% \end{equation}

%The column vectors in $\Lambda(J)$ are denoted by $\lambda_{1}(J), \ldots,
%\lambda_{m+j-1}(J)$ or, conveniently by  $ \widetilde{\lambda}_{1}, \ldots, \widetilde{\lambda}_{m+j-1}$. The %corresponding characteristic classes of $X(J)$ satisfy
%$\widetilde{w}_1 = \widetilde{w}_{m+i}$ for $ i = m+1, \ldots, m+j-1$.
 
%It is not hard to argue by induction that  for any $n$-dimensional cone $<\lambda_{i_1}, \ldots, \lambda_{i_n}> $ of $\Sigma$, the cohomology class $  \widetilde{w}_{1}^{j-1} \widetilde{w}_{i_1} \cdots  \widetilde{w}_{i_n} $ is a generator of $H^{2(m+ j-1)} (X(J))$.

%Denote the homology class of the submanifold given by the transversal intersection of the  divisors $\widetilde{D}_{m+1}, \ldots, \widetilde{D}_{m+j-1}$ in $H_{\ast}(X(J))$ by $V$. For any $n$-dimensional cone 
%$<\lambda_{i_1}, \ldots, \lambda_{i_n}> $ of
%$\Sigma$, the intersection number $ V \cdot \widetilde{D}_{i_1} \cdots \widetilde{D}_{i_n} $ is  one. 

 \begin{defn} Let $X$ be a nonsingular toric variety with fan $\Sigma$.
	Suppose  $w_p $ in $H^2(X)$ is the cohomology class corresponding to the torus invariant divisor $D_p$. We say that $w_p$ has the {\it isolation property} if it 	admits a decomposition 	$w_p = \sum_{k \in I} a_k w_k$, where $ p \notin I$, such that $\lambda_p$ and $\lambda_k$ do not form a cone in $\Sigma$ for  each $k \in I$.    \end{defn}

\begin{lemma}\label{wsquare} Let $X$ be a nonsingular toric variety.
	Suppose $w_p$ admits the isolation property. Then $w_p^2 = 0$.  
\end{lemma}

\begin{proof} Consider a decomposition 	$w_p = \sum_{k \in I} a_k w_k$, where $ p \notin I$, such that $\lambda_p$ and $\lambda_k$ do not form a cone in $\Sigma$ for  each $k \in I$.
	Note that 	$w_p^2 =  \sum_{k \in I} a_k w_k w_p$, and each $w_k w_p = 0$ as $\lambda_k, \lambda_p$ do not form a cone. \end{proof}

\begin{lemma}\label{isoJ} Let $X$ be a smooth toric variety. Let $X(J)$ be a nonsingular toric variety  obtained from $X$ by a $J$-construction.
	Suppose $w_p \in H^2(X)$ admits the isolation property. 
	Then $w_p(J)$ also admits the isolation property.
\end{lemma}	

\begin{proof} By induction, it is enough to verify this for an atomic $J$-construction. Without loss of generality, let 
	$J = (2, 1, \ldots, 1)$.  Then, $$ \sum a_i w_i = 0 \implies
	 \sum a_i w_i(J) = 0 \,.$$ 
	Moreover, if $\lambda_i$ and $\lambda_j$ do not form a cone in $\Sigma$, neither do 
	$\lambda_i(J)$ and $\lambda_j(J)$ in $\Sigma(J)$.
	The lemma follows.
\end{proof}

\begin{corollary}\label{corsq}
	 Suppose $X$ is a Bott manifold $M_k$.	
	Let  $ X(J)$ be a nonsingular toric variety  obtained from $X$ by a $J$-construction.
	Then there exist nontrivial line bundles on $X(J)$ 
	whose first Chern class squared is zero. 
\end{corollary}

\begin{proof} Note that the class $w_p$ in a Bott manifold $M_k$, where $p\le k$, has the isolation property if  the constants $c_{1,p}, \ldots, c_{p-1, p}$ are zero (see \eqref{weqs}) as $\lambda_p^{(k)}$ and $\lambda_{p+k}^{(k)}$ do not form a cone. In particular, the class $w_1$ in a Bott manifold always has the isolation property. Then the result follows from Lemmas \ref{wsquare} and \ref{isoJ}. 	
\end{proof}

\section{Construction of SKT metrics}\label{metric}

Suppose $\pi: E \to X$ is principal 
torus bundle over a complex manifold $X$ with a $2k$-dimensional real torus $T \cong (S^1)^{2k}$ as fiber. 
The Lie algebra $\mathfrak{t}$ of $T$ has a natural lattice given by the circle subgroups of $T$.  A choice of an integral basis of $\mathfrak{t}$ gives  a decomposition of
$T$ into a product of $S^1$'s.

Let $\Theta$ be a connection $1$-form on the bundle $E \to X$  and   $(\theta_1, \ldots, \theta_{2k})$ be a representation of $\Theta$ corresponding to a  basis of $\mathfrak{t}$. Define an almost complex structure $\mathcal{J}_E$ on $E$ by lifting the complex structure $\mathcal{J}_X$ on $X$ to the horizontal space of $\Theta$, and by defining $$\mathcal{J}_E \,\theta_{2j-1} = \theta_{2j}, \; 1\le j\le k \,,$$ along the vertical directions. The Chern classes of $\pi: E \to X$ are generated by $\omega_j \in \Omega^2(X)$ where $\pi^{\ast}(\omega_j) = d\theta_j$. If the classes
 $w_j := [\omega_j] \in H^2(X, \mathbb{R})$ are of type $(1,1)$, then  $\mathcal{J}_E$ is integrable and $\pi$ is holomorphic with respect to it, see \cite[Lemma 1]{ggp}.

 If the  bundle $\pi: E \to X$ is obtained by reduction of structure group from a holomorphic principal bundle over $X$, then the above construction of integrable complex structure on $E$ is applicable, see \cite[Section 5]{PT}. 
In the sequel, we will  apply this
construction when the torus bundle is obtained by reduction from a direct sum of holomorphic line bundles. 
 
 Suppose that $g_X$ is a Hermitian metric on $X$ with fundamental form $F_X$. Following \cite{ggp}, consider a 
 Hermitian metric $g_E$ on $E$ defined by $$g_E = \pi^{\ast} g_X + \sum_{j=1}^{2k} \theta_j \otimes \theta_j \,.$$ The fundamental form of $g_E$ is given by 
 $$F_E = \pi^{\ast} F_X + \sum_{j=1}^{k} \theta_{2j-1} \wedge \theta_{2j} \,.   $$
 It follows that 	    
$$ dd^c F_E =  \pi^{\ast}dd^c F_X - \sum_{j=1}^{2k} \pi^{\ast}( \omega_j \wedge \omega_j)\,.
$$
Thus $E$ admits an SKT structure if 
\begin{equation}\label{skt}
dd^c F_X = \sum_{j=1}^{2k}  \omega_j \wedge \omega_j \,.	
\end{equation}	

Therefore, if $X$ is K\"ahler, the vanishing of $\sum \omega_j \wedge \omega_j  $
is sufficient for existence of an SKT structure on $E$. The following lemma, which is based on a trick in the proof of Theorem 15 of \cite{ggp}, shows that the vanishing of
$\sum [\omega_j] \wedge [\omega_j] = \sum w_j^2  $ is sufficient.

\begin{lemma}\label{suf}
	If $X$ is K\"ahler and $\sum_{j=1}^{2k} w_j^2 = 0$, then $E$ admits an SKT  structure.
\end{lemma}

\begin{proof}
As $\sum w_j^2 = 0$,
 the (2,2)-form $\sum \omega_j\wedge \omega_j $ is $d$-exact and $d$-closed. The complex structure  $\mathcal{J}_X$ on $X$  preserves $(2,2)$-forms. Hence, $\sum \omega_j \wedge \omega_j $
is $d^c$-closed. So, by the $dd^c$-lemma, there exists a real $(1,1)$-form $\alpha$ on $X$ such that $\sum \omega_j \wedge \omega_j = dd^c\alpha $.
  Choose an appropriate multiple $\beta$ of a K\"ahler form on $X$ such that
$$\min_{p\in X}(\min_{|| \zeta || =1} \beta_p( \zeta, \mathcal{J}_X \zeta) ) > - \min_{p\in X}(\min_{|| \zeta || =1} \alpha_p( \zeta, \mathcal{J}_X \zeta) )\,.$$ 
Then the positive definite form $\alpha+\beta$ defines a Hermitian metric $g_X$ on $X$ which satisfies \eqref{skt}.  
\end{proof}
	
\begin{theorem}\label{tbun} Suppose $X$ is a Bott manifold $M_k$.	
	Let  $ X(J)$ be a nonsingular toric variety  obtained from $X$ by a $J$-construction.
	 Then there exist nontrivial principal torus bundles $E(J)$ over $X(J)$ such that the total space of $E(J)$ admits an SKT structure.   
\end{theorem}	

\begin{proof}   As noted in Corollary \ref{corsq} ,  there exist nontrivial line bundles on $X(J)$ 
	whose first Chern class squared is zero. (This also follows  from Lemma \ref{linesquare} when $X(J) =X$.)
	 Consider $E(J)$ to be the torus bundle over $X(J)$ obtained by the reduction of structure group from the direct sum of an even number of such  line bundles. 
As $X$ is projective,  $X(J)$ is also projective (cf. \cite[Theorem 2]{ewald}), and hence K\"ahler.  
Then $E(J)$ admits an SKT structure by Lemma \ref{suf}.
\end{proof}	

Note that the above procedure produces a family of SKT metrics on $E(J)$ parametrised by an open subset of the
 space of K\"ahler metrics on $M_k$. A good reference for K\"ahler metrics on Bott manifolds is the recent article \cite{BCT}
 by Boyer et al. 

Compact connected Lie groups of even dimension admit invariant complex structures (cf. \cite{pittie} or \cite[Section 2]{PT}). Let $G$ be such a group  whose rank  $ \ge 2k$. Then there exists an injective homomorphism of Lie groups, $\phi: T \to G$ such that the image of $\phi$ is a closed subgroup of $G$.  A choice of a complex structure on $T$ has been made 
above while defining $\mathcal{J}_E$. Assume that we have used an integral basis of $\mathfrak{t}$ in defining $\Theta$. It is then explained in \cite[Section 5]{PT} how an invariant complex 
structure may be chosen on $G$ so that the map $\phi$ is holomorphic. The following result then follows from the proof of \cite[Theorem 5.2]{PT}.  

\begin{corollary}
Let $E(J)$ be a principal $T$-bundle as in Theorem \ref{tbun}, and let $\phi: T \to G$ be a holomorphic
monomorphism of Lie groups as above. Then the total space of the principal $G$-bundle 
${\displaystyle E(J) \times^\phi G}$ admits an SKT structure.  
\end{corollary}

%As an example, consider the Hirzebruch surface $F_k$ with 
%$$\Lambda = \left[ \begin{array}{cccc}
%1 & 0 & -1 & 0 \\
%0 & 1 & k & -1  \end{array} \right] \, $$
%Then $w_3 = w_1$ and $w_4 = w_2 + k w_3$. As $w_1 w_3 = w_2 w_4 = 0$, we easily obtain 
%$w_2^2 + w_4^2 = 0$. Let $\widetilde{X}$ be the toric variety obtained by performing the 
%$hc-construction on $X$ with $J = (j,1,\ldots, 1)$. Then the torus bundle $E$ on %$\widetilde{X}$ with first Chern class $\widetilde{w}_2 + \widetilde{w}_4$ admits an SKT 
%structure.

\section*{Acknowledgements}
%\noindent % We thank Ajay Singh Thakur for useful conversations.

 The research of the last named author has been supported in part by a SERB MATRICS Grant, MTR/2019/001613, and an SRP grant from Middle East Technical University Northern Cyprus Campus. He is grateful to Anthony Bahri for introducing him to the $J$-construction. He also thanks Ajay Singh Thakur for useful conversations and collaboration on related topics.

\bibliographystyle{spmpsci}

\end{document}